\documentclass[a4paper,draft,leqno,12pt]{article}

\usepackage{amsmath,bm}
\usepackage{amscd}
\usepackage{amssymb}
\usepackage{enumerate}
\usepackage{indentfirst}
\usepackage{latexsym}
\usepackage{multicol}
\usepackage{theorem}

\newcommand{\AAA}{{\cal A}}

\makeatletter\@addtoreset{equation}{section}\makeatother

\newtheorem{definition}{Definition}
\newtheorem{lemma}{Lemma}{\theorembodyfont{\rmfamily}
\newtheorem{remark}{Remark}}\newtheorem{theorem}{Theorem}

\newenvironment{proof}{\textit{Proof. }}{\hfill$\Box$}

\newcommand{\dx}{{\rm d} x}

\newcommand{\dt}{{\rm d} t }

\newcommand{\ep}{\varepsilon}

\newcommand{\Z}{\mbox{\F Z}}

\newcommand{\ds}{\displaystyle}
\newcommand{\beq}[1]{\begin{equation} \label{#1}\ds}
\newcommand{\eeq}{\end{equation}}
\newcommand{\bml}[1]{\beq{#1} \begin{array}{c}\ds}
\newcommand{\eml}{\end{array}\eeq}
\newcommand{\beqq}{\begin{equation*}\ds}
\newcommand{\eeqq}{\end{equation*}}
\newcommand{\bmll}{\beqq \begin{array}{c}\ds}
\newcommand{\emll}{\end{array}\eeqq}
\renewcommand{\div}{{\rm div}\,}

\newcommand{\abs}[1]{\ensuremath{\left| #1 \right|}}

\def \fii{\varphi}
\def \de{\partial}

\def \Om{\Omega}

\def \ep{\varepsilon}

\def \AA{\mathcal{A}}
\def \AAA{\mathbb{A}}
\def \MM{\mathcal{M}}

\def \I{\mathbf{I}}

\newcommand{\R}{\mathbb{R}}

\def \Z{\mathbf{Z}}

\def \sil{\rightarrow}

\title{$\AA$-free Rigidity and Applications to the Compressible Euler System}
\author{Elisabetta Chiodaroli, Eduard Feireisl,\\ Ond\v{r}ej Kreml, and Emil Wiedemann}

\begin{document}

\maketitle

\begin{abstract}
Can every measure-valued solution to the compressible Euler equations be approximated by a sequence of weak solutions? 
We prove that the answer is negative: Generalizing a well-known rigidity result of Ball and James to a more general situation, we construct
an explicit measure-valued solution for the compressible Euler equations which can not be generated by a sequence of distributional solutions. We also give an abstract necessary condition for measure-valued solutions to be generated by weak solutions, relying on work of Fonseca and M\"uller. 

This difference between weak and measure-valued solutions in the compressible case is in contrast with the incompressible situation, where every measure-valued solution can be approximated by weak solutions, as shown by Sz\'ekelyhidi and Wiedemann.
\end{abstract}

\section{Introduction}\label{s:I0}
The compressible Euler equations 
\begin{equation}\label{eq:CE}
\begin{aligned}
 \de_t \rho + \div_x (\rho v) &= 0 \\
 \de_t (\rho v) + \div_x (\rho v \otimes v) + \nabla_x p(\rho) &= 0
\end{aligned}
\end{equation}
describe the motion of a perfect fluid whose state is determined by the distribution of the
fluid velocity $v$ and of the mass density $\rho$. The pressure $p$ is given as a function of the density 
and satisfies $p(\rho)\geq0$, $p'(\rho)>0$.
These equations were first formulated by Euler in 1757, but we are still far from a complete understanding 
of the phenomena described by the Euler equations.

If we consider the Cauchy problem for the compressible Euler equations \eqref{eq:CE} on $Q = [0,T] \times\Omega$ (for a, say, bounded smooth domain $\Omega\subset\R^3$) with, say, bounded initial data
\begin {equation*}
(\rho,v)(0,\cdot) = (\rho_0,v_0)
\end{equation*}
then we say that a pair $(\rho, v)\in L^1 ([0,T]\times \Omega)$ is a \textit{weak solution} to \eqref{eq:CE} (together with a natural boundary condition $v \cdot n = 0$ on $\partial\Omega$, $n$ being the unit outer normal to $\partial\Omega$) if it satisfies
the following equations:
\begin{equation*}
\int_0^T \int_\Omega \de_t \psi \rho + \nabla_x\psi \cdot \rho v \dx\dt + \int_\Omega \psi(0,x) \rho_0(x) \dx = 0,
\end{equation*}
\begin{align*}
&\int_0^T \int_\Omega \de_t\fii \cdot \rho v + \nabla_x\fii : \rho v \otimes v + \div_x\fii p(\rho) \dx\dt \\
\nonumber + &\int_\Om \fii(0,x)\cdot \rho_0(x) v_0(x) \dx = 0
\end{align*}
for all $\psi \in C^\infty_c([0,T) \times \Omega)$ and all $\fii\in C^\infty_c([0,T) \times \Omega;\R^3)$ (in particular, it is part of the definition that all the integrals exist, which is certainly the case for solutions with finite energy).
The concept of weak solutions allows to deal with shock formation which may occur even for smooth data; however weak solutions are well known to be non-unique.
In the literature many admissibility criteria have been proposed to restore uniqueness and the most classical one relies on suitable entropy conditions.

Entropy solutions have been widely accepted as the suitable solution framework for
systems of conservation laws in several space dimensions, for which the compressible Euler equations \eqref{eq:CE} are the most paradigmatic example. However, in space dimensions greater than one the existence of entropy weak solutions is not known, so that the weaker concept of measure-valued solution was introduced by Neustupa~\cite{Neu}, who also proved the existence of such solutions. Previously, the notion of measure-valued solutions for systems of conservation laws had been introduced by DiPerna in the seminal paper \cite{Di}. Measure-valued solutions are only Young measures, instead of integrable functions. 

On the one hand, this allows to characterize complex phenomena such as oscillations and concentrations. On the other hand, measure-valued solutions have been criticized for being too weak, as is apparent from their obvious non-uniqueness (but see the weak-strong uniqueness results in~\cite{BrDLSz, GwSwWi}). However, recent results by the first three authors jointly with De Lellis \cite{Ch,ChDlKr, ChFeKr}
have demonstrated that for the compressible Euler equations \eqref{eq:CE} even entropy solutions (weak solutions satisfying a suitable entropy inequality) may not be unique, 
thus raising anew the problem of a correct notion of solutions for \eqref{eq:CE}.

In the recent paper \cite{FKMT}, the authors show that current numerical schemes may not necessarily converge to an entropy solution of systems of conservation laws 
in several space dimensions, when the mesh is refined.
Instead, entropy measure-valued solutions are suggested as 
an appropriate solution paradigm for systems of conservation laws, on the base of a large number of numerical experiments.

In the context of hyperbolic conservation laws in one space dimension, when the equation is complemented with a suitable entropy condition, the Tartar-Murat theory of compensated compactness applies (see \cite{Di2})
and oscillations in weakly converging sequences of solutions can be excluded. The situation is very different for the incompressible Euler equations in several space dimensions: In \cite{SzWi}, Sz\'ekelyhidi and the last author showed that any measure-valued solution
can be generated by a sequence of exact solutions. This means that, in any space dimension, measure-valued solutions and weak solutions are substantially the same, thus leading to a very large set of weak solutions.

Motivated by this series of results, we aimed at understanding the corresponding problem for the compressible Euler equations \eqref{eq:CE} (we treat here only the three-dimensional case).
This issue may prove interesting both in view of the solution theory of the compressible Euler system \eqref{eq:CE} itself and
for possible selection criteria among measure-valued solutions of incompressible Euler based on low Mach number limits.
In this paper we prove (Theorem~\ref{t:main1}) that any measure-valued solution to the compressible Euler system which arises 
as a limit of a sequence of
weak solutions must satisfy a Jensen-type inequality, which is reminiscent of analogous inequalities in the so called $\AA$-free setting (cf. \cite{FoMu} and \cite{KiPe, KiPe2}). To our knowledge this is the first time that such a Jensen inequality appears in the context of fluid dynamics.

Moreover, we construct
an explicit measure-valued solution for the compressible Euler equations which can
not be generated by a sequence of distributional solutions (Theorem~\ref{t:main2}). 
Thus we show a substantial difference between weak and measure-valued solutions for the compressible Euler equations \eqref{eq:CE} 
in contrast with the situation for incompressible Euler \cite{SzWi}. 
Our construction relies on a rigidity result for so-called $\AA$-free sequences (Theorem~\ref{lemma1} below) in the spirit of compensated compactness~\cite{Ta, Mu}, which generalizes a well-known result of Ball and James (Proposition 2 in~\cite{BaJa}, see also Lemma 2.2 in~\cite{Mue}) on sequences of gradients. Rigidity results of this type play an important role in nonlinear elasticity, but have not yet been used in fluid mechanics. Our generalization to the $\AA$-free setting will not come as a surprise to experts, but we have not been able to find it explicitly stated in the literature. Under the constant rank assumption, our Theorem~\ref{lemma1} follows from Theorem 4 in~\cite{Ri}, but we do not require this assumption and our proof is simple and self-contained. We thus hope it will be found useful independently of our specific application in the context of the Euler equations.

Finally, let us briefly comment on the converse of Theorem~\ref{t:main1}: Is the Jensen inequality sufficient for a measure-valued solution to be generated by weak solutions? In the incompressible case, this was shown to be true~\cite{SzWi}, but the proof relies on convex integration techniques which do not seem to be flexible enough to transfer to the compressible case.  

The paper is organized as follows: Sections~\ref{sec:mvs},~\ref{sec:Afree} and~\ref{sec:wave} contain preliminary material on measure-valued solutions and linear differential constraints. In Section~\ref{s:robustness} we state and prove our $\AA$-free rigidity theorem. Section 4 then gives the mentioned applications of the general theory to the compressible Euler equations.   

\textbf{Acknowledgments.} The authors would like to thank L\'aszl\'o Sz\'ekelyhidi Jr.\ for helpful discussions. 
The research of E.\ F. leading to these results has received funding from the European Research Council under the European Union's Seventh Framework Programme (FP7/2007-2013)/ ERC Grant Agreement 320078. The Institute of Mathematics of the Academy of Sciences of the Czech
Republic is supported by RVO:67985840.
The work of O.\ K. was supported by grant of GA\v CR (Czech Science Foundation) GA13-00522S in the general framework of RVO:67985840.

\section{Measure-valued solutions and subsolutions}\label{sec:mvs}
We denote by $\MM(\R^d)$ the space of finite Radon measures on $\R^d$ and by $\MM^1(\R^d)$ the subset of probability measures.
For a subset $\Omega\subset \R^N$, we also denote by $L^\infty_{w} (\Omega; \MM^1(\R^d))$ the space of weakly$^*$-measurable maps from $\Omega$ into $\MM^1(\R^d)$, meaning that for each Borel function $f: \R^d\rightarrow \R$, the map
\begin{equation*}
x\mapsto \langle\nu_x,f\rangle := \int_{\R^d} f(z){\rm d}\nu_x(z)
\end{equation*}
is measurable (with respect to Lebesgue measure).\\
A \textit{Young measure} on a set $\Omega\subset \R^N$ is a map $\nu\in L^\infty_{w} (\Omega; \MM^1(\R^d))$ which assigns to almost every point $x \in \Omega$ 
a probability measure $\nu_{x}\in \MM^1(\R^d)$ on the phase space $\R^d$.
We note in passing that $\langle\nu_x,f\rangle := \int_{\R^d} f(z){\rm d}\nu_x(z)$ is the expectation of $f$ with respect to the probability measure $\nu_{x}$.

In the context of measure-valued solutions to evolution equations, one of the variables is interpreted as time and the domain takes the form $\Omega=[0,T]\times\tilde{\Omega}$ for some space domain $\tilde{\Omega}$ (we will omit the tilde as long as no confusion is to be expected). 
In particular, for the three-dimensional Euler equations, we have $N=1+3$ and $d=4$. We will then denote the state variables by $\xi\in \R^4$ and more precisely we introduce the following notation in order to formalize the definition of measure valued solution to \eqref{eq:CE}:
\begin{align*}
\xi &= [\xi_0,\xi'] = [\xi_0,\xi_1, \xi_2, \xi_3] \in \R^+ \times \R^3 \\
\langle\nu_{t,x},\xi_0\rangle &= \overline{\rho} \\
\langle\nu_{t,x}, \sqrt{\xi_0}\xi'\rangle &= \overline{\rho v} \\
\langle\nu_{t,x}, \xi'\otimes\xi'\rangle &= \overline{\rho v \otimes v} \\
\langle\nu_{t,x}, p(\xi_0)\rangle &= \overline{p(\rho)}.
\end{align*}
One should think of $\xi_0$ as the state of the density $\rho$ and $\xi'$ as the state of $\sqrt{\rho}v$.

\begin{definition}[Measure-valued solution]
A \emph{measure-valued solution to the compressible Euler equations} \eqref{eq:CE} is a Young measure $\nu_{t,x}$ on $\R^+ \times \R^3$ with parameters in $[0,T] \times \Omega$ which satisfies the Euler equations in an average sense, i.e.
\begin{equation*}
\int_0^T \int_\Omega \de_t \psi \overline{\rho} + \nabla_x\psi \cdot \overline{\rho v} \dx\dt + \int_\Omega \psi(0,x) \rho_0(x) \dx = 0
\end{equation*}

\begin{align*}
&\int_0^T \int_\Omega \de_t\fii \cdot \overline{\rho v} + \nabla_x\fii : \overline{\rho v \otimes v} + \div_x\fii \overline{p(\rho)} \dx\dt \\
\nonumber + &\int_\Om \fii(0,x)\cdot \rho_0(x) v_0(x) \dx = 0
\end{align*}
for all $\psi \in C^\infty_c([0,T) \times \Omega)$ and all $\fii\in C^\infty_c([0,T) \times \Omega;\R^3)$.
\end{definition}
Again, it is understood that all the integrals have to be well-defined and finite. Note that this definition involves only the ``classical" Young measure and ignores effects of concentration (confer the notion of ``generalized" Young measure as established in~\cite{DPMa}). This is mainly for reasons of simplicity, and because our counterexample (Theorem~\ref{t:main2}) excludes an approximating sequence even if it forms concentrations (cf.\ Remark~\ref{concentration} below).

Observe that every measurable function $u: \Omega \rightarrow \R^d$ naturally gives rise to a Young measure by defining $\nu_x:= \delta_{u(x)}$: such Young measures are called \textit{atomic}. 
Thus every weak solution of \eqref{eq:CE} defines an atomic measure-valued solution.

If $\{z_n\}$ is a sequence of measurable functions $\Omega\to\R^d$ we say that $\{z_n\}$ \textit{generates} the Young measure $\nu$ if for all bounded Carath\'eodory functions $f: \Omega \times \R^d \sil \R$
\begin{equation*}
\lim_{n\sil\infty} \int_{\Omega} f(x,z_n(x)) \varphi(x) \dx = \int_{\Om} \langle \nu_x,f(x,\cdot)\rangle \varphi(x)  \dx
\end{equation*}
for all $\varphi \in L^1(\Omega)$.
By the fundamental theorem of Young measures any suitably bounded sequence of Young measures has a converging subsequence in the above sense. In particular, any sequence of functions bounded in $L^p(\Omega)$ (for any $p\geq1$) generates, up to a subsequence, some Young measure.

Our first result gives a necessity condition, in the form of a Jensen-type inequality, for a measure-valued solution to \eqref{eq:CE} in order to be generated by a sequence of weak solutions.
This Jensen-type inequality involves a suitable notion of subsolutions for our system \eqref{eq:CE}.
More precisely we deal with solutions to the related linear system
\begin{equation}\label{eq:LS}
\begin{aligned}
 \de_t \rho + \div_x m &= 0 \\
 \de_t m + \div_x U + \nabla_x q &= 0,
\end{aligned}
\end{equation}
which we call \textit{subsolutions}. The linear system \eqref{eq:LS} is obtained from \eqref{eq:CE} by first changing the state variable from $(\rho, v)$ to $(\rho, m:=\rho v)$, where $m$ is the 
linear momentum, and then by replacing every nonlinearity appearing in
the original compressible Euler equations \eqref{eq:CE} by a new variable: in particular $U\in S^3_0 $ is a symmetric trace-free $3\times 3$ matrix which replaces the traceless 
part of the matrix $\rho v \otimes v= \frac{m\otimes m }{\rho}$ and $q$ encodes the pressure plus the term coming from the trace of $\rho v \otimes v$ (see \cite{DLSz, ChDlKr} for similar definitions).
Weak solutions to \eqref{eq:LS} are functions $(\rho, m, U, q)$ which satisfy \eqref{eq:LS} in the sense of distributions.

Accordingly we have to define measure-valued subsolutions. Similarly as above, we use the following notation:
\begin{align*}
[\zeta_0,\zeta',\Z,\tilde{\zeta}] &\in \R^+ \times \R^3 \times S^3_0 \times \R^+ \\
\langle\mu_{t,x},\zeta_0\rangle &= \overline{\rho} \\
\langle\mu_{t,x}, \zeta'\rangle &= \overline{m} \\
\langle\mu_{t,x}, \Z\rangle &= \overline{U} \\
\langle\mu_{t,x}, \tilde{\zeta}\rangle &= \overline{q}
\end{align*}

\begin{definition}[Measure valued subsolution]
A \emph{measure-valued solution to the linear system} is a Young measure $\mu_{t,x}$ on $\R^+ \times \R^3 \times S^3_0 \times \R^+$ with parameters in $[0,T] \times \Omega$ 
which satisfies the linear system \eqref{eq:LS} in an average sense, i.e.
\begin{equation*}
\int_0^T \int_\Omega \de_t \psi \overline{\rho} + \nabla_x\psi \cdot \overline{m} \dx\dt + \int_\Omega \psi(0,x) \rho_0(x) \dx = 0
\end{equation*}

\begin{align*}
&\int_0^T \int_\Omega \de_t\fii \cdot \overline{m} + \nabla_x\fii : \overline{U} + \div_x\fii \overline{q} \dx\dt \\
\nonumber + &\int_\Om \fii(0,x)\cdot m_0(x) \dx = 0
\end{align*}
for all $\psi \in C^\infty_c([0,T) \times \Omega)$ and all $\fii\in C^\infty_c([0,T) \times \Omega;\R^3)$.
\end{definition}

We also define, in analogy with~\cite{SzWi}, a lift of a Young measure from the space of solutions to the space of subsolutions. Indeed once we are given a measure-valued solution
to the compressible Euler equations \eqref{eq:CE} this defines naturally a corresponding measur--valued subsolution to the linear system \eqref{eq:LS}.

\begin{definition}[Lift]
Let $\nu_{t,x}$ be a measure valued solution to the Euler equations. Denote $Q: \R^+ \times \R^3 \mapsto \R^+ \times \R^3 \times S^3_0 \times \R^+$
\begin{equation*}
Q(\xi) := (\xi_0,\sqrt{\xi_0}\xi',\xi'\otimes\xi' - \frac 13 \abs{\xi'}^2\I,p(\xi_0) + \frac 13 \abs{\xi'}^2).
\end{equation*}
We define the \emph{lifted measure} $\tilde{\nu}_{t,x}$ as
\begin{equation*}
\langle\tilde{\nu}_{t,x},f\rangle := \langle\nu_{t,x},f \circ Q\rangle
\end{equation*}
for $f \in C_0(\R^+ \times \R^3 \times S^3_0 \times \R^+)$ and a.e. $(t,x)$.
\end{definition}

\section{Linearly constrained sequences}
\subsection{$\AA$--free setting and constant rank property}\label{sec:Afree}

The linear system \eqref{eq:LS} associated to the compressible Euler equations fits nicely into the so-called $\AA$-free framework for linear partial differential constraints, introduced by Tartar~\cite{Ta}. Consider a general linear system of $l$ differential equations in $\R^N$ written as
\begin{equation}\label{eq:operator def}
\AA z := \sum_{i=1}^N A^{(i)} \frac{\de z}{\de x_i} = 0,
\end{equation}
where $A^{(i)}$ ($i = 1,...,N$) are $l \times d$ matrices and $z: \R^N \sil \R^d$ is a vector-valued function. Next, we define the $l \times d$ matrix
\begin{equation*}
\AAA(w) := \sum_{i=1}^N w_i A^{(i)}
\end{equation*}
for $w \in \R^N$.

\begin{definition}[Constant rank]
We say that $\AA$ has the \emph{constant rank property} if there exists $r \in \mathbb{N}$ such that
\begin{equation*}
{\rm rank}\, \AAA(w) = r
\end{equation*}
for all $w \in \mathcal{S}^{N-1}$.
\end{definition}



The following theorem from \cite{FoMu} will be the cornerstone of our proof of Theorem~\ref{t:main1} below. In order to state it, we first recall the general definition of $\AA$-quasiconvex functions.

\begin{definition}[$\AA$-Quasiconvexity]
A function $f:\R^{d} \sil \R$ is said to be $\AA$-\emph{quasiconvex} if
\begin{equation}\label{eq:Aqc}
f(z) \leq \int_{(0,1)^N} f(z + w(x)) \dx
\end{equation}
for all $z \in \R^d$ and all $w\in C^{\infty}_{per}((0,1)^N;\R^d)$ such that $\AA w = 0$ and $\int_{(0,1)^N} w(x) \dx = 0$.
\end{definition}
Recall that a sequence $\{z_n\}$ is called \emph{$p$-equiintegrable} if the sequence $\{|z_n|^p\}$ is equiintegrable in the usual sense.
\begin{theorem}[Theorem 4.1 in \cite{FoMu}]\label{t:T1}
Let $1 \leq p < \infty$ and let $\{\nu_x\}_{x\in\Om}$ be a weakly measurable family of probability measures on $\R^d$. Let $\AA$ have the constant rank property. There exists a $p$-equi-integrable sequence $\{z_n\}$ in $L^p(\Omega;\R^d)$ that generates the Young measure $\nu$ and satisfies $\AA z_n = 0$ in $\Omega$ if and only if the following conditions hold:
\begin{itemize}
\item[(i)] there exists $z \in L^p(\Om;\R^d)$ such that $\AA z = 0$ and $z(x) = \langle\nu_x,{\rm id}\rangle$ a.e. $x \in \Om$;
\item[(ii)] \begin{equation*}
\int_{\Om} \int_{\R^d} \abs{w}^p {\rm d}\nu_x(w) \dx < \infty;
\end{equation*}
\item[(iii)] for a.e. $x \in \Om$ and all $\AA$-quasiconvex functions $g$ that satisfy $\abs{g(w)} \leq C(1+\abs{w}^p)$ for some $C > 0$ and all $w \in \R^d$ one has
\begin{equation}\label{eq:Jensen general}
\langle\nu_x,g\rangle \geq g(\langle\nu_x,{\rm id}\rangle).
\end{equation}
\end{itemize}
\end{theorem}




\subsection{Wave analysis} \label{sec:wave}

When considering first order linear PDEs in the form 
\eqref{eq:operator def}, a special class of solutions are \textit{plane waves}, i.e.\ solutions of the form $z(x)=h(x\cdot\xi)\bar{z}$ with $h:\R \to \R$.
In order to identify such solutions, one has to solve the relation $\sum_{i=1}^N  \xi_i A^{(i)} \bar{z} = 0$ which gives rise to the following definition.
\begin{definition}
Consider a linear differential operator $\AA$ as in~\eqref{eq:operator def}. Its \emph{wave cone} $\Lambda$ is defined as the set 
of all $\bar{z}\in\R^d\setminus\{0\}$ for which there exists $\xi\in\R^N\setminus\{0\}$ such that
\begin{equation*}
z(x)=h(x\cdot\xi)\bar{z}
\end{equation*}
satisfies $\AA  z =0$ for any choice of profile function $h: \R\to\R$.

Equivalently, $\bar{z}\in\Lambda$ if and only if $\bar{z}\neq0$ and there exists $\xi\in\R^N\setminus\{0\}$ such that $\AAA (\xi)\bar{z}=0$.
\end{definition}
The claimed equivalence can be easily verified by taking the Fourier transform of $\AA z=0$. It can be seen immediately that $\Lambda$ is relatively closed in $\R^d\setminus\{0\}$.
In other words, the wave cone $\Lambda$ characterizes the directions of one dimensional oscillations compatible with \eqref{eq:operator def}.

Let us give an equivalent reformulation of the constraint~\eqref{eq:operator def} and accordingly yet another characterization of the wave cone. Observe that
\begin{equation*}
\begin{aligned}
\sum_{i=1}^NA^{(i)}\frac{\partial z}{\partial x_i}&=\left(\sum_{i=1}^N\sum_{k=1}^dA_{jk}^{(i)}\frac{\partial z_k}{\partial x_i}\right)_{j=1,\ldots,l}\\
&=\left(\sum_{i=1}^N\frac{\partial}{\partial x_i}\sum_{k=1}^dA_{jk}^{(i)}z_k\right)_{j=1,\ldots,l}.
\end{aligned}
\end{equation*}
Therefore, if we define the $l\times N$-matrix $Z_\AA$ by
\begin{equation}\label{eq:defZ}
(Z_\AA)_{ji}=\sum_{k=1}^dA_{jk}^{(i)}z_k,\qquad j=1,\ldots,l,\qquad i=1,\ldots,N,
\end{equation}
then~\eqref{eq:operator def} can be rewritten as
\begin{equation}\label{eq:divfree}
\div Z_\AA=0.
\end{equation}
Moreover, the condition $\AAA(\xi)\bar{z}=0$ from the definition of the wave cone translates to $\bar{Z}_{\AA} \xi=0$ (where $\bar{Z}_\AA$ is obtained from $\bar{z}$ via~\eqref{eq:defZ}), so that the following are equivalent: 
\begin{enumerate}
\item $\bar{z}\in\Lambda$;
\item $\bar{z}\neq0$ and ${\rm rank}\, \bar{Z}_\AA<N$.
\end{enumerate}
It follows immediately that $\Lambda=\R^d\setminus\{0\}$ if $l<N$.

\subsection{$\AA$-free rigidity} \label{s:robustness}

In this section we prove a generalization of the well-known rigidity result of Ball and James~\cite{BaJa} to the $\AA$-free framework. Besides its possible independent interest, its application to the system~\eqref{eq:operator def} will guarantee a form of compensated compactness that allows us to construct a measure-valued solution to the compressible Euler system which is not generated by a sequence of weak solutions. 

\begin{theorem} \label{lemma1}

Let $\Omega\subset\R^N$ be a domain, $\AA$ a linear operator of the form~\eqref{eq:operator def}, and $1<p<\infty$. Let moreover $z_n : \Omega \to \R^{d}$ be a family of functions with
\[
\| z_n \|_{L^p(\Omega; \R^{d})} \leq c,
\]
\begin{equation} \label{H1}
\AA z_n=0 \ \mbox{in}\ \mathcal{D}'(\Omega),
\end{equation}
and suppose $(z_n)$ generates a compactly supported Young measure
\[
\nu_{x} \in \mathcal{M}^1 (\R^{d})
\]
such that
\begin{equation} \label{H2}
{\rm supp} [ \nu_{x} ] \subset \{ \lambda \bar{z}_1 + (1 - \lambda) \bar{z}_2),\ \lambda \in [0,1] \} \ \mbox{for a.a. } x \in \Omega
\end{equation}
and for some given constant states $\bar{z}_1, \bar{z}_2 \in \R^{d}$, $\bar{z}_1\neq \bar{z}_2$.
Suppose that
\[
\bar{z}_2 - \bar{z}_1 \not \in \Lambda.
\]
Then
\[
z_n \to z_\infty \ \mbox{ in} \ L^p(\Omega),
\]
which implies that
\[
\nu_{x} = \delta_{z_{\infty}(x)} , \ z_{\infty}(x) \in \{ \lambda \bar{z}_1 + (1 - \lambda) \bar{z}_2),\ \lambda \in [0,1] \} \ \mbox{for a.a.}\ x \in \Omega.
\]
More specifically, $z_\infty$ is a constant function of the form
\[
z_\infty = \overline{\lambda} \bar{z}_1 + (1 - \overline{\lambda}) \bar{z}_2.
\]
for some fixed $\overline{\lambda} \in [0,1]$.

\end{theorem}

\begin{remark}
Note that we are assuming neither $p$-equiintegrability of the sequence nor the constant rank property. If, however, the constant rank property holds, then hypothesis~\eqref{H1} can be relaxed to 
\begin{equation*} 
\AA z_n\to 0 \ \mbox{in}\ W^{-1,p}(\Omega)
\end{equation*}
using Lemma 2.14 from~\cite{FoMu}.
\end{remark}

The rest of the Section is devoted to the proof of Theorem \ref{lemma1} which will be split in two steps corresponding
to as many subsections.

\subsubsection{The specific form of $z_n$}

Let us start with a preliminary remark. 
Since we already know that
\[
z_n \rightharpoonup z_\infty \ \mbox{weakly in}\ L^p(\Omega; \R^{d}),
\]
it is enough to show that $\{z _n \}_{n \geq 1}$ contains a subsequence converging in $L^p(\Omega)$.

Moreover, as the result is local, it is enough to show the claim on any set
\[
\Omega_\delta = \{x\in\Omega: {\rm dist} [ x, \partial \Omega ] > \delta\},\ \delta > 0.
\]
In other words, we may assume, without loss of generality, that
(\ref{H1}) holds in an open set containing $\overline{\Omega}$.

Now, we aim to prove the following claim.\\ \\
CLAIM: The functions $z_n$ have the following specific form
\begin{equation} \label{claim1}
z_n(x) = {e}_n(x) + \lambda_n (x) \bar{z}_1 + (1 - \lambda_n(t,x) ) \bar{z}_2 ,
\end{equation}
where
\[
\ e_n \to 0 \ \mbox{in}\ L^p(\Omega; \R^{d}) \;  \mbox{as } n\to\infty,
\]
and $\lambda_n$ are bounded measurable functions
\[
0 \leq \lambda_n \leq 1 \ \mbox{a.e. in}\ \Omega.
\]

\begin{proof}[Proof of the CLAIM]
Let $\delta > 0$ be an arbitrary positive constant. Then there is a function
\[
F_\delta \in C^\infty (\R^{d}),\ 0 \leq F_\delta \leq 1 
\]
such that
\[
F_\delta = 0 \ \mbox{in an open neighborhood of the segment}\ \left\{ \lambda \bar{z}_1 + (1 - \lambda) \bar{z}_2,\ \lambda \in [0,1] \right\}
\]
\[
F_{\delta}(\bar{z}) = 1 \ \mbox{if}\ \ {\rm dist} \left[ \bar{z} ; \left\{ \lambda \bar{z}_1 + (1 - \lambda) \bar{z}_2,\ \lambda \in [0,1] \right\} \right] > \delta .
\]
Next, we can simply rewrite $z_n$ as
\[
z_n = F_{\delta} (z_n) z_n + (1 - F_\delta (z_n ) )z_n.
\]
Of course, we have
\[
\left\| F_{\delta} (z_n) z_n \right\|_{L^p(\Omega; \R^{d})} \leq \left\| z_n \right\|_{L^p(\Omega; \R^{d})},
\]
and, by virtue of hypothesis \eqref{H2}, we also have
\[
\left\| F_{\delta} (z_n) z_n \right\|_{L^p(\Omega; \R^{d})} \to 0 \ \mbox{as}\ n \to \infty, \ \mbox{for any}\ \delta > 0.
\]

Finally, we can write
\[
(1 - F_\delta (z_n) (x) ) z_n(x) =
\]
\[
(1 - F_\delta (z_n) (x) ) \left( z_n(x) - \lambda_{n,\delta} (x) \bar{z}_1 - (1 - \lambda_{n,\delta} (x) ) \bar{z}_2 \right)
\]
\[
- F_\delta (z_n) (x)  \left( \lambda_{n,\delta} (x) \bar{z}_1 + (1 - \lambda_{n,\delta} (x) ) \bar{z}_2 \right) 
\]
\[
 + \left( \lambda_{n,\delta} (x) \bar{z}_1 + (1 - \lambda_{n,\delta} (x) ) \bar{z}_2 \right)
\]
for certain
\[
0 \leq \lambda_{n,\delta} \leq 1,
\]
where
\[
\left| (1 - F_\delta (z_n) (x) ) \left( z_n(x) - \lambda_{n,\delta} (x) \bar{z}_1 - (1 - \lambda_{n,\delta} (x) ) \bar{z}_2 \right) \right|
\leq \delta
\]
and
\[
\left\| F_\delta (z_n) (x) ) \left( \lambda_{n,\delta} (x) \bar{z}_1 + (1 - \lambda_{n,\delta} (x) ) \bar{z}_2 \right) \right\|_{L^p(\Omega; \R^{d})} \to 0
\ \mbox{as}\ n \to \infty
\]
for any fixed $\delta > 0$.

As $\delta > 0$ can be taken arbitrarily small, \eqref{claim1} and the claim follows.
\end{proof}

\subsubsection{Constraint imposed by equations}

We use hypothesis \eqref{H1} in the form~\eqref{eq:divfree} taking advantage of the specific form of $z_n$ proved in (\ref{claim1}). We employ here the map~\eqref{eq:defZ} from vectors in $\R^d$ to matrices in $\R^{l\times N}$ and denote by a capital letter the value of the map at the corresponding lowercase letter, e.g.\ $e_n$ is mapped to $E_n$ according to~\eqref{eq:defZ}. 

Hence by~\eqref{H1} we have
\[
0 =
\div  E_n(x)  + \div  [\lambda_n(x) ( \bar{Z}_1 - \bar{Z}_2 ) ]
\]
\[
=\div  E_n(x)  -(\bar{Z}_2 - \bar{Z}_1 )  \nabla \lambda_n(x) \ \mbox{in}\ \mathcal{D}'(\Omega).
\]

Now, using our preliminary remark that the equations are satisfied in an open neighborhood of $\overline{\Omega}$, we may use a regularization $v \mapsto (v)_\ep$ by means of a suitable
family of convolution kernels to obtain
\begin{equation} \label{pom1}
 (\bar{Z}_2 - \bar{Z}_1) \nabla (\lambda_n)_\ep  = \div  (E_n)_\ep \equiv \chi_{n, \ep} ,
\end{equation}
where
\begin{equation} \label{pom2}
\| \chi_{n, \ep} \|_{W^{-1,p}(\Omega; \R^{d})} \leq c_1(n) \to \ 0 \ \mbox{as} \ n \to \infty \ \mbox{independently of}\ \ep.
\end{equation}
Next we make use of the fact that $\bar{z}_2 - \bar{z}_1 \not \in \Lambda$, and hence $\bar{Z}_2 - \bar{Z}_1 $ has full rank (recall the discussion in Section~\ref{sec:wave}) and thus possesses a left inverse. Then it follows from \eqref{pom1}--\eqref{pom2} that
\[
\left\|  \nabla (\lambda_n - \left< \lambda_n \right> )_\ep \right\|_{W^{-1,p}(\Omega; \R^{d})} \leq c_2(n)\to 0;
\]
whence
\[
\left\| \nabla (\lambda_n - \left< \lambda_n \right> )\right\|_{W^{-1,p}(\Omega; \R^{d})} \leq c_2(n);
\]
and finally (by Ne\v cas' Lemma),
\[
\left\| \lambda_n - \left< \lambda_n \right> \right\|_{L^p(\Omega; \R^{d})} \to 0 \ \mbox{as}\ n \to \infty
\]
where
$$\langle \lambda_n\rangle:= \frac{1}{|\Omega|} \int_\Omega \lambda_n(x)  \dx.$$
Going back to \eqref{claim1} we obtain the desired strong convergence of $\left\{ z_n \right\}_{n \geq 0}$ to a constant function. 
We have proved Theorem \ref{lemma1}.

\section{Application to the compressible Euler equations}\label{sec:application}
\subsection{A necessary condition for generability by weak solutions}\label{sec:proof1}

With $t = x_0$ the linear system \eqref{eq:LS} related to subsolutions of the compressible Euler equations can be rewritten in the form \eqref{eq:operator def}. More precisely, keeping in mind that $U$ is a symmetric traceless $3\times 3$ matrix, we define the state vector
\begin{equation*}
z := (\rho,m_1,m_2,m_3,U_{11}, U_{12}, U_{13}, U_{22}, U_{23},q) \in \R^{10}.
\end{equation*}
Accordingly, the $4 \times 10$ matrices $A^{(i)}_L$ for $i=0,...,3$ have the following form
\begin{equation}\label{eq:A0}
A^{(0)}_L = \left(
\begin{array}{cccccccccc}
	1 & 0 & 0 & 0 & 0 & 0 & 0 & 0 & 0 & 0 \\
	0 & 1 & 0 & 0 & 0 & 0 & 0 & 0 & 0 & 0 \\
	0 & 0 & 1 & 0 & 0 & 0 & 0 & 0 & 0 & 0 \\
	0 & 0 & 0 & 1 & 0 & 0 & 0 & 0 & 0 & 0
\end{array}
\right)
\end{equation}
\begin{equation}
A^{(1)}_L = \left(
\begin{array}{cccccccccc}
	0 & 1 & 0 & 0 & 0 & 0 & 0 & 0 & 0 & 0 \\
	0 & 0 & 0 & 0 & 1 & 0 & 0 & 0 & 0 & 1 \\
	0 & 0 & 0 & 0 & 0 & 1 & 0 & 0 & 0 & 0 \\
	0 & 0 & 0 & 0 & 0 & 0 & 1 & 0 & 0 & 0
\end{array}
\right)
\end{equation}
\begin{equation}
A^{(2)}_L = \left(
\begin{array}{cccccccccc}
	0 & 0 & 1 & 0 & 0 & 0 & 0 & 0 & 0 & 0 \\
	0 & 0 & 0 & 0 & 0 & 1 & 0 & 0 & 0 & 0 \\
	0 & 0 & 0 & 0 & 0 & 0 & 0 & 1 & 0 & 1 \\
	0 & 0 & 0 & 0 & 0 & 0 & 0 & 0 & 1 & 0
\end{array}
\right)
\end{equation}
\begin{equation}\label{eq:A3}
A^{(3)}_L = \left(
\begin{array}{cccccccccc}
	0 & 0 & 0 & 1 & 0 & 0 & 0 & 0 & 0 & 0 \\
	0 & 0 & 0 & 0 & 0 & 0 & 1 & 0 & 0 & 0 \\
	0 & 0 & 0 & 0 & 0 & 0 & 0 & 0 & 1 & 0 \\
	0 & 0 & 0 & 0 & -1 & 0 & 0 & -1 & 0 & 1
\end{array}
\right).
\end{equation}

With this notation, the system \eqref{eq:LS} takes the form $\AA_{L} z := \sum_{i=0}^3 A^{(i)}_L \frac{\de z}{\de x_i} = 0$. We can prove the following lemma.

\begin{lemma}\label{l:constant rank}
The operator $\AA_L$ defined as $\AA_{L} z := \sum_{i=0}^3 A^{(i)}_L \frac{\de z}{\de x_i}$,
 with the choice of matrices $A^{(i)}_L$ given by \eqref{eq:A0}--\eqref{eq:A3}, has the constant rank property with $r = 4$.
\end{lemma}
\begin{proof}
This follows from an easy linear algebra computation.
\end{proof}

The constant rank property is crucial to us since it allows to exploit the theory of Fonseca and M\" uller for our linear system \eqref{eq:LS}. Recall once more that a sequence $\{z_n\}$ is called \emph{$p$-equiintegrable} if the sequence $\{|z_n|^p\}$ is equiintegrable. Here, $p$-equi-integrability refers to both variables $t$ and $x$. Then, by virtue of Theorem \ref{t:T1} and Lemma~\ref{l:constant rank}, we can finally prove the following theorem, which is one of our main results:

\begin{theorem}\label{t:main1}
Suppose the pressure function satisfies $c\rho^\gamma\leq p(\rho)\leq C\rho^\gamma$ for some $\gamma\geq 1$ and $\{(\rho_n,v_n)\}$ is a sequence of weak solutions to the compressible Euler system \eqref{eq:CE} such that $\{\rho_n\}$ is $\gamma$-equiintegrable and $\{\sqrt{\rho_n}v_n\}$ is 2-equiintegrable. Suppose moreover $\{(\rho_n,\sqrt{\rho_n}v_n)\}$ generates a Young measure $\nu$ on $\R^+\times\R^3$. Then $\nu$ is a measure-valued solution to the compressible Euler system \eqref{eq:CE} and the lifted measure $\tilde{\nu}$ on $\R^+\times\R^3\times\mathcal{S}^3_0\times\R^+$ satisfies
\begin{equation}\label{eq:Jensen3}
\langle \tilde{\nu}_{t,x}, g \rangle \geq g(\langle \tilde{\nu}_{t,x},{\rm id}\rangle)
\end{equation}
 for all $\AA_{L}$-quasiconvex functions $g$.
\end{theorem}

\begin{remark}
The assumption on the pressure is very natural. Together with the equiintegrability assumptions it implies that the nonlinear terms $\{\rho_n v_n\otimes v_n\}$ and $\{p(\rho_n)\}$ do not concentrate, so that our classical Young measure framework is applicable. The equiintegrability assumption, however, is more difficult to justify, since it is not implied by the usual apriori estimates for~\eqref{eq:CE} related to the energy inequality. We use this assumption here for convenience in order to not overly complicate the presentation; the interested reader is referred to the paper~\cite{FoKr}, which extends the Fonseca-M\"uller theory to non-equiintegrable sequences. An application of this extension to the Euler system then implies a result similar to Theorem~\ref{t:main1} including concentration effects. We omit details.
\end{remark}
\begin{remark}
Condition~\eqref{eq:Jensen3} is rather abstract, as it involves the little understood concept of $\AA_L$-quasiconvexity. It is not even obvious whether this condition is vacuous or not. Indeed, for the incompressible Euler equations, one can obtain a similar Jensen inequality, but it turns out that the wave cone is so large that the corresponding $\AA$-quasiconvexity already implies convexity, so that~\eqref{eq:Jensen3} is tautologically satisfied. The main point of this paper is to demonstrate that this is not the case for the compressible system, which is accomplished in Subsection~\ref{sec:proof2} below. 
\end{remark}

\begin{proof}
The proof proceeds through the following steps:
\begin{itemize}
\item[Step 1:] The fact that the Young measure $\nu$ is a measure valued solution to the compressible Euler equations \eqref{eq:CE} is a direct consequence of the Fundamental Theorem of Young measures and the lack of concentrations implied by the equiintegrability assumptions.
\item[Step 2:] The sequence $\{(\rho_n,v_n)\}$ naturally gives rise to a sequence of weak 
solutions $\{z_n\} = \{\rho_n,m_n,U_n,q_n\}$ to the linear system \eqref{eq:LS}, by defining 
$$m_n = \rho_n v_n,\quad U_n := \rho_n v_n \otimes v_n - \frac 13 \rho_n\abs{v_n}^2\I, \quad q_n := p(\rho_n) + \frac 13 \rho_n\abs{v_n}^2$$
where $\I$ is the $3\times3$ identity matrix.
\item[Step 3:] We show that the lifted measure $\tilde{\nu}$ is generated by the sequence $\{z_n\}$. This is an easy consequence of the definition of the lifted measure and the fact that $\nu$ is generated by $\{(\rho_n,\sqrt{\rho_n}v_n)\}$:
\begin{align*}
\int_0^T\int_\Omega\langle\tilde{\nu}_{t,x},g(t,x,\cdot)\rangle \varphi(x,t) \dx\dt &= \int_0^T\int_\Omega\langle\nu_{t,x},(g\circ Q)(t,x,\cdot)\rangle \varphi(x,t)\dx\dt \\
 &= \lim_{n \sil \infty}\int_0^T\int_\Omega g(t,x,Q(\rho_n,\sqrt{\rho_n}v_n)(t,x)) \varphi(x,t) \dx\dt \\
&= \lim_{n \sil \infty}\int_0^T\int_\Omega g(t,x,z_n(t,x))  \varphi(x,t) \dx\dt
\end{align*}
for all test functions $\varphi$.
\item[Step 4:] Applying Theorem \ref{t:T1} with the choice $\mu=\tilde{\nu}$, we obtain that $\tilde{\nu}$ is a measure valued solution to the linear system \eqref{eq:LS} and it satisfies the Jensen inequality \eqref{eq:Jensen3}.
\end{itemize}
\end{proof}

\subsection{An explicit example}\label{sec:proof2}


In this final section we aim to show that there exists a measure-valued solution of~\eqref{eq:CE} which is not generated by a sequence of weak solutions, which also shows that the Jensen condition~\eqref{eq:Jensen3} is not vacuous. To this end, let us study the wave cone for our linear system \eqref{eq:LS}.
By virtue of~\eqref{eq:defZ}, to each state vector 
\begin{equation*}
z := (\rho,m_1,m_2,m_3,U_{11}, U_{12}, U_{13}, U_{22}, U_{23},q) \in \R^{10} 
\end{equation*}
we associate the $4\times 4$ matrix $Z_{\AA_L}$ given as
\begin{equation} \label{eq:matrixZ}
 Z_{\AA_L} = \left[ \begin{array}{cccc} 
\rho & m_1 & m_2 & m_3 \\ 
m_1 & U_{11} + q & U_{12} & U_{13} \\ 
m_2 & U_{12} & U_{22} + q & U_{23} 
\\ m_3 & U_{13} & U_{23} & - U_{11}- U_{22} + q 
\end{array} \right].
\end{equation}
Hence, the wave cone for the operator $\AA_L$ is equal to
$$\Lambda_L= \Big\{ \bar{z} \in \R^{10} \text{ such that }\det(\bar{Z}_{\AA_L})=0\Big\}$$
Moreover, by~\eqref{eq:divfree}, the linear system \eqref{eq:LS} can be written 
$$\div_{t,x} Z_{\AA_L} =0.$$

Finally, using Theorem \ref{lemma1} and an explicit construction, we can prove the existence of a measure-valued solution
to \eqref{eq:CE} which is not generated by weak solutions:

\begin{theorem}\label{t:main2}
There exists a measure-valued solution of the compressible Euler system \eqref{eq:CE} 
which is not generated by any sequence of $L^p$-bounded weak solutions to \eqref{eq:CE} (for any choice of $p>1$).
\end{theorem}
\begin{remark}\label{concentration}
\begin{enumerate}
\item Any reasonable sequence of approximate solutions of~\eqref{eq:CE} will satisfy some uniform energy bound, so that the assumption of $L^p$-boundedness will always be met.
\item As Theorem~\ref{lemma1} did not require any equiintegrability, the statement of Theorem~\ref{t:main2} is true even when the potential generating sequence is allowed to concentrate. In the language of generalized Young measures, this means that there exists a generalized measure-valued solution which can not be generated by a sequence of weak solutions (take the measure from Theorem~\ref{t:main2} as the oscillation part and choose the concentration part arbitrarily).
\end{enumerate}
\end{remark}
\begin{proof}
The idea is to choose suitable $z^1:=(\rho^1,m^1,U^1,q^1)$ and $z^2:=(\rho^2,m^2,U^2,q^2)$ such that the homogeneous Young measure
\begin{equation*}
\tilde{\nu}=\frac{1}{2}\delta_{z^1}+\frac{1}{2}\delta_{z^2}
\end{equation*}
cannot be a limit of bounded weak solutions to the linear system \eqref{eq:LS}.

Set $\rho^1=1$, $m^1=e_1$, $U^1=\operatorname{diag}(2/3,-1/3,-1/3)$, $q^1=p(1)+\frac{1}{3}$ and, 
for some $\gamma>0$, $\rho^2=\gamma$, $m^2=e_1$, $U^2=U^1/\gamma$, $q^2=p(\gamma)+\frac{1}{3}\gamma$. 
Notice that with this choice of $z^1$ and $z^2$, $\tilde{\nu}$ arises as the lifted Young measure of some measure-valued solution $\nu$ to~\eqref{eq:CE}; indeed, this follows from
\begin{equation*}
U^1=\frac{m^1\otimes m^1}{\rho^1}-\frac{|m^1|^2}{\rho^1},\qquad q^1=p(\rho^1)+\frac{|m^1|^2}{\rho^1}
\end{equation*}
and similarly for $z^2$.

Now, observe that $\tilde{z}:=z^2-z^1$ is not in the wave cone for the operator $\AA_L$. 
Indeed, forming the corresponding $4\times4$-matrix $\tilde{Z}$ according to~\eqref{eq:defZ},
a direct calculation yields that, with our choice of $z^1$ and $z^2$, the determinant of the matrix $\tilde{Z}$ is
\begin{equation*}
\left(1-\frac{1}{\gamma}+p(1)-p(\gamma)\right)(p(1)-p(\gamma))^2,
\end{equation*}
and by the assumptions on the pressure (recall $p\geq0$, $p'>0$) it is readily seen that $\gamma$ can be chosen in such a way that this is nonzero, i.e. $\tilde{z}\notin\Lambda_L$.

Finally, by virtue of Theorem \ref{lemma1} applied with the choice $\bar{z}_1:= z^1$ and $\bar{z}_2:= z^2$, we easily see that $\tilde{\nu}$ cannot be generated
by any sequence of subsolutions. Moreover, as noted before, $\tilde{\nu}$ arises as the lifting of a measure valued solution $\nu$
to the original compressible Euler equations \eqref{eq:CE} of the form 
$$\nu=\frac{1}{2}\delta_{\left(\rho^1, \frac{m ^1}{\sqrt{\rho^1}}\right)}+\frac{1}{2}\delta_{\left(\rho^2, \frac{m ^2}{\sqrt{\rho^2}}\right)}$$
which, as a consequence, cannot be generated by any sequence of weak solutions to \eqref{eq:CE}, since this would contradict what we have shown at the level of subsolutions.
This proves Theorem \ref{t:main2}.


\end{proof}

\end{document}